\documentclass[a4paper,10pt]{amsart}
\usepackage{amsmath,amsthm,amssymb,tikz,calc,graphicx}
\usepackage[latin1]{inputenc}
\usepackage[T1]{fontenc}
\usepackage[all]{xy}
\usepackage{mathrsfs,pifont}
\usepackage{booktabs}
\usepackage{array}
\usepackage{longtable}
\usepackage{placeins}
\usepackage{enumitem,hyperref,geometry}

\DeclareMathOperator{\Ei}{Ei}
\DeclareMathOperator{\ERi}{ERi}
\newcommand{\C}{\mathbb C}

\newtheorem{theorem}[equation]{Theorem}
\newtheorem*{thm*}{Main Theorem}

\newtheorem{lemma}[equation]{Lemma}
\newtheorem{proposition}[equation]{Proposition}

\theoremstyle{definition}

\newtheorem{defn-lem}[equation]{Definition-Lemma}
\newtheorem{remark}[equation]{Remark}
\newtheorem*{rem*}{Remark}

\numberwithin{equation}{section}

\usepackage{bm}
\usepackage{upgreek}
\makeatletter
\newcommand{\bfgreek}[1]{\bm{\@nameuse{up#1}}}

\def\R{\mathbb R}
\def\N{\mathbb N}

\def\C{\mathbb C}

\def\<{\langle}
\def\>{\rangle}

\makeatletter
\DeclareRobustCommand\smallop[2][1]{%
	\mathop{\vphantom{\oplus}\mathpalette\smallop@{{#1}{#2}}}\slimits@
}
\newcommand{\smallop@}[2]{\smallop@@#1#2}
\newcommand{\smallop@@}[3]{%
	\vcenter{%
		\sbox\z@{$#1\oplus$}%
		\hbox{\resizebox{\ifx#1\displaystyle#2\fi\dimexpr\ht\z@+\dp\z@}{!}{$\m@th#3$}}%
	}%
}
\makeatother

\makeatletter
\DeclareRobustCommand\bigop[2][1]{%
	\mathop{\vphantom{\bigoplus}\mathpalette\bigop@{{#1}{#2}}}\slimits@
}
\newcommand{\bigop@}[2]{\bigop@@#1#2}
\newcommand{\bigop@@}[3]{%
	\vcenter{%
		\sbox\z@{$#1\bigoplus$}%
		\hbox{\resizebox{\ifx#1\displaystyle#2\fi\dimexpr\ht\z@+\dp\z@}{!}{$\m@th#3$}}%
	}%
}
\makeatother

\begin{document}

\title[On divergence related to Riemann--von Mangoldt's explicit formula for $\pi(x)$]{On divergence related to Riemann--von Mangoldt's explicit formula of the prime-counting function}
\author{Harald Grobner}
\thanks{The author has been supported by the research projects PAT-4628923 and PAT-2584625 of the Austrian Science Fund (FWF), the EU NextGeneration project (IIP\_UNIPU\_010159) and the Croatian Science Foundation project (HRZZ-IP-2022-10-4615).}
\address{Harald Grobner, Faculty of Mathematics, University of Vienna, Oskar-Morgenstern-Platz 1, A-1090 Vienna, Austria}
\email{harald.grobner@univie.ac.at}

\date{\today}

\begin{abstract}
An explicit formula for the prime-counting function $\pi(x)$, usually attributed to Riemann and von Mangoldt, is prominently stated as the equation $\pi(x)=R(x)-\sum_\rho R(x^\rho)$, where the sum runs over all zeros $\rho$ of the Riemann $\zeta$-function, the non-trivial ones being ordered by increasing absolute value of their imaginary parts and counted with multiplicity. This particularly entails the claim that the partial sums over the non-trivial zeros, $\Sigma R_T(x):=\sum_{0<|\Im m(\rho)|\le T} R(x^\rho)$ converge as ${T\to\infty}$. Writing $\Theta:=\sup\{\Re e(\rho):\ \zeta(\rho)=0,\ 0<\Re e(\rho)<1\},$ for what has recently been called ``Riemann's constant'', we prove that, for every fixed $x>1$ and every $\theta<\Theta$, the sums $\Sigma R_T(x)$ are not $O(T^\theta)$.  As a consequence, $\limsup_{{T\to\infty}}|\Sigma R_T(x)|=\infty$ and $\sum_\rho R(x^\rho)$ diverges. We conclude the paper by showing that an adapted, but simpler strategy also gives the analogous result -- and hence divergence -- for the contribution of the trivial zeros to $\sum_\rho R(x^\rho)$. In both cases the exponent $\Theta$ is proved to be sharp.
\end{abstract}
\maketitle

\vskip 10pt

\setcounter{tocdepth}{2}
\tableofcontents

\section*{Introduction}
\subsection*{Context} Consider the exponential integral function $\Ei(s)$ for $s\in\C^*$ and its sibling $\ERi(s):=\sum_{n\geq 1} \frac{\mu(n)}{n} \Ei(\frac{s}{n})$, which we continuously complete to an entire function by $\ERi(0):=1$. If $x>0$ is real and $\rho$ denotes any zero of the Riemann $\zeta$-function, one usually abbreviates $R(x^\rho):=\ERi(\rho\log(x))$. In his famous note, \cite{Zagier1977}, p.\ 14, Zagier writes Riemann's prime-counting function (at its points of continuity) in the form
\begin{equation*}
\pi(x)=R(x)-\sum_{\rho}R(x^\rho).
\end{equation*}
where the sum is taken over all zeros $\rho$ of the Riemann $\zeta$-function. As this sum is not absolutely convergent, the usual precise meaning given to it is to sum the non-trivial zeros, which all lie discrete and with finite multiplicity $m_\rho$ in the closed critical strip $0\leq\Re e(z)\leq1$, ordered pairwise by increasing absolute value of their imaginary parts and counted with their multiplicity and the trivial zeros ordered along $-2,-4,-6,-8,...$, i.e., 
\begin{equation}\label{eq:Rdef2}
\sum_{\rho}R(x^\rho):= \lim_{T\to\infty}\sum_{\substack{\zeta(\rho)=0\\ 0<|\Im m(\rho)|\le T}} m_\rho R(x^\rho) + \sum_{k\geq 1} R(x^{-2k}).
\end{equation}
Including the points of discontinuity of $\pi(x)$, the statement can be formulated uniformly for all $x>1$ by considering the averaged prime-counting function and by claiming that
\begin{equation}\label{eq:1}
\pi_0(x)= R(x) - \sum_{\rho} \ERi(\rho \log(x)),
\end{equation}
summation $\sum_{\rho}$ being as defined above. The status of formula \eqref{eq:1} has attracted caution in the literature, however. While it is certainly true that for $x>1$
\begin{equation}\label{eq:1'}
\pi_0(x)= R(x) - \sum_{n\geq 1} \sum_{\rho} \frac{\mu(n)}{n} \Ei\left(\frac{\rho\log(x)}{n}\right),
\end{equation}
an identity which one may call the {\it Riemann--von Mangoldt explicit formula for $\pi_0(x)$}, the claim of \eqref{eq:1}, namely -- after having reinserted into the definition of $\ERi(s)$ --  that also
$$
\pi_0(x)= R(x) - \sum_{\rho} \sum_{n\geq 1} \frac{\mu(n)}{n} \Ei\left(\frac{\rho\log(x)}{n}\right),
$$
holds, i.e., that the two limits underlying the two sums may be interchanged, has remained unproven. However, this amounts to a non-trivial and in fact delicate question, since the conditional convergence of the two series involved on the right hand side of \eqref{eq:1'} would make such a proof non-trivial, cf.\ \cite{BorweinBradleyCrandall2000}, p.\ 249. For a more systematic discussion of this issue, we refer in particular to the recent book of Elliott, \cite{Elliott2025}, especially \S 5.1 and Rem.\ 5.1.2 therein, which includes also a review of the often claimed identity
\begin{equation}\label{eq:trivRs}
\sum_{k\geq 1} R(x^{-2k})=\frac{1}{\log(x)} - \frac{1}{\pi} \arctan{\frac{\pi}{\log(x)}} 
\end{equation}
for the second summand in  \eqref{eq:Rdef2}.

\subsection*{Results} In this note we show that the convergence of the series over the non-trivial zeros of the Riemann $\zeta$-function, showing up as the first summand in \eqref{eq:Rdef2}, in fact does not hold. More precisely, for $x>1$ and $T>0$ put
$$
\Sigma R_T(x):= \sum_{\substack{\zeta(\rho)=0\\0<|\Im m(\rho)|\le T}} m_\rho\ERi(\rho\log x).
$$
In other words, $\Sigma R_T(x)$ is the (finite) sum of the values $\ERi(\rho\log x)$, summation being over all non-trivial zeros $\rho$ of $\zeta(s)$, whose imaginary part is bounded above in absolute value by $T$, and counted with their (finite) multiplicity $m_\rho$. Now set 
$$\Theta:=\sup\{\Re e(\rho):\ \zeta(\rho)=0,\ 0<\Re e(\rho)<1\},$$
which was called ``Riemann's constant'' in  \cite{Elliott2025}. As is well-known, one has $1/2\leq\Theta\leq 1$, the Riemann hypothesis being equivalent to $\Theta=1/2$. The following is our main result.

\begin{thm*}\label{thm:main}
For every fixed $x>1$ and every $\theta$ with $\theta<\Theta$, the partial sums $\Sigma R_T(x)$ are not $O(T^\theta)$ as $T\to\infty$. In particular, for every $x>1$, $\limsup_{T\to\infty}|\Sigma R_T(x)|=\infty$ and $\Sigma R_T(x)$ diverges. 
\end{thm*}

\noindent It follows from this result, that the expression $\sum_{\rho}R(x^\rho)$ of \eqref{eq:Rdef2} is not a well-defined complex number and hence an equation of the form \eqref{eq:1} cannot hold with the usual meaning of summation over the zeros of $\zeta(s)$ as recalled above.\\\\
We sketch the main idea of the proof of our theorem and refer for all details to the body of the paper below. The starting point is the integral representation
$$
\ERi(z)=\int_0^1 \sum_{n\leq u^{-1}}\frac{\mu(n)}{n} u^{-1}e^{zu}\,du,
$$
cf.\ Prop.\ \ref{prop:kernel}. For a fixed $A>0$, this turns
$$
f_A(y):=\Re e(\ERi(-A+iy))
$$
into the Fourier cosine transform of an $L^1((0,1))$-function. Its truncated Mellin transform can then be continued meromorphically and the continuation inherits poles from the non-trivial zeros of $\zeta(s)$, cf.\ Prop.\ \ref{prop:HA}. We show that these poles translate directly into a growth obstruction for the logarithmically weighted primitive function of $f_A(y)$, 
$$
P_A(Y)=\int_1^Y \log(yC^{-1})\,f_A(y)\,dy.
$$
More precisely, if $\rho$ is any non-trivial zero of $\zeta(s)$, then $P_A(Y)$ cannot be $O(Y^\theta)$ for any $\theta<\Re e(\rho)$. See Prop.\ \ref{prop:Omega} for details.\\\\
The second part of the argument transfers this result into an obstruction on the partial sums $\Sigma R_T(x)$. We integrate
$$
\ERi(\log(x)s)\left(-\frac{\zeta'(s)}{\zeta(s)}\right)
$$
over the contour given by a rectangle $\Delta_T(a,c)$, whose vertical sides lie in $\Re(s)=c>1$ and $\Re(s)=-a<0$ and whose height $T>0$ is not the ordinate $\Im m(\rho)$ of a non-trivial zero $\rho$.  The residue theorem produces the difference of $\ERi(\log(x))$ with the partial sum $\Sigma R_T(x)$ over the non-trivial zeros, cf.\ Prop.\ \ref{lem:residue}.\\\\
On the other hand, the right vertical side of this contour-integral recovers the averaged prime-counting function $\pi_0(x)$, see Prop.\ \ref{lem:right}; while on the left vertical side the functional equation for $\zeta(s)$, together with Stirling's formula, isolates $P_{a\log(x)}(T\log x)$ as the only term capable of unbounded growth -- all remaining contributions have finite limits as $T\to\infty$. It is our choice of standard good ordinates that ensures that the contribution of the horizontal sides of the rectangle $\Delta_T(a,c)$ are negligible in the limit. See Lem.\ \ref{lem:good} and Prop.\ \ref{prop:left}\\\\
Thus, along these good ordinates $T$, avoiding the imaginary parts of all non-trivial zeros of $\zeta(s)$, while $T\to\infty$, 
$$
\Sigma R_T(x) = \frac{1}{\pi\log x}\,P_{a\log(x)}(T\log x)+D_{a,x}+o(1).
$$
for a constant $D_{a,x}$ independent of $T$. Since good ordinates may be chosen within distance $O(1)$ of every sufficiently large $T$, a bound $\Sigma R_T(x) =O(T^\theta)$ would contradict the above growth obstruction for $P_{a\log(x)}(T)$, whenever $\theta<\Re e(\rho)$ for some non-trivial zero $\rho$ of $\zeta(s)$. The reader will find this final argument eventually carried out in Sect.\ \ref{sect:proof}.\\\\
We also remark that the bound provided in our Main Theorem is in fact sharp, as $\Sigma R_T(x) = O_{x,\varepsilon}(T^{\Theta+\varepsilon})$ for every $\varepsilon>0$ as $T\to\infty$. The reader will find the short argument underlying this assertion in our Rem.\ \ref{rem:exact-power-scale}.\\\\
We conclude our paper with some results, which deal with the contribution of the trivial zeros, written as $\sum_{k\geq 1} R(x^{-2k})$ in \eqref{eq:Rdef2}. The divergence of the corresponding series can be established by suitably adapting, while substantially simplifying the above sketched strategy underlying the proof of our Main Theorem: More precisely, we show that for any given $x>1$ and $\theta<\Theta$ 
$$\Sigma R^{triv}_T(x):=\sum_{1\leq k\leq T} \ERi(-2k\log(x))$$
is not $O(T^\theta)$ as $T\to\infty$. In particular, $\limsup_{T\to\infty}|\Sigma R^{triv}_T(x)|=\infty$ and hence also this series diverges. We refer the reader to Thm.\ \ref{prop:trivial-zeros} in Sect.\ \ref{sect:41} for a proof.\\\\
Again, it is important to note that the provided bound is sharp, as one has $\Sigma R^{triv}_T(x)=O_{x,\varepsilon}(T^{\Theta+\varepsilon})$ for every $\varepsilon>0$ as $T\to\infty$. This turns out to be much simpler to see than the analogous assertion for $\Sigma R_T(x)$ and we refer the reader to Rem.\ \ref{rem:trivsharp}.\\\\ 
On a final note, we also observe that our results, taken together, also yield the divergence of the ``merged'' sum $$\Sigma R^*_T(x):=\sum_{\substack{\zeta(\rho)=0\\ |\rho-\frac12|\leq T}}m_\rho\,\ERi(\rho\log(x)),$$
as $T\to\infty$, where trivial and non-trivial zeros are now summed all together according to their distance to $1/2$. We refer to Rem.\ \ref{rem:final} for this final observation.

\small  

\normalsize

\section{Non-trivial zeros and growth}
\subsection{Notational preliminaries}
We use the standard Vinogradov- and Landau-notation. Thus, if $f$ is a complex-valued and $g$ is non-negative real-valued function, then $f(T)\ll g(T)$ means that there exists a constant $C>0$ such that $|f(T)|\le C\cdot g(T)$ for all sufficiently large $T$. Equivalently, $f(T)=O(g(T))$. If we want to indicate that the implied constant is allowed to depend on one or more fixed parameters, we indicate this by subscripts. For example, $f(T)\ll_{a,b,c} g(T)$ means that $|f(T)|\le C_{a,b,c}\cdot g(T)$ for all sufficiently large $T$, where $C_{a,b,c}>0$ may depend on ${a,b,c}$, but not on $T$. We write synonymously $f(T)=O_{a,b,c}(g(T))$.

\subsection{A kernel representation of \texorpdfstring{$\ERi$}{ERi}} 
First, we recall the definition of the exponential integral function: For $s\in\C\setminus [0,\infty)$ it is given by
$$\Ei(s):=\int_{-\infty}^s \frac{e^z}{z}\,dz$$
where integration is along any path not crossing $[0,\infty)$, and then extended to $s\in\C^*$ by setting
$$\Ei(s):=\lim_{\varepsilon\to 0}\frac{\Ei(x+\varepsilon i)+\Ei(x-\varepsilon i)}{2}$$
for $s\in (0,\infty)$. Moreover, we let
$$\ERi(s):=\sum_{n\geq 1} \frac{\mu(n)}{n} \Ei\left(\frac{s}{n}\right),$$ 
continuously completed by $\ERi(0):=1$ (in the above series and in what follows, $\mu$ denotes the M\"obius function). Equivalently, $\ERi(s)$ may be defined as the unique entire function, which is $1$ at $s=0$ and whose $k$-th derivatives satisfy $\ERi^{(k)}(0)=\frac{1}{k \zeta(k+1)}$ for all $k\geq 1$, i.e., it admits a Taylor series expansion of the form
\begin{equation}\label{eq:ERI}
\ERi(s)=1+\sum_{k=1}^{\infty}
\frac{s^k}{k\,k!\,\zeta(k+1)}.
\end{equation}
See \cite{Elliott2025}, \S 4.4-4.5.\\\\
For $y>0$, resp. $0<u<1$, we define, cf.\ \cite{MontgomeryVaughan}, (6.19),
$$
B(y):=\sum_{n>y}\frac{\mu(n)}{n}=-\sum_{n\le y}\frac{\mu(n)}{n}, \qquad {\rm and }\qquad Q(u):=u^{-1}B(u^{-1}).
$$
It follows from \cite{MontgomeryVaughan}, (6.18), that $B(y)\ll \log(y)^{-2}$. Hence, using that $B(y)$ is bounded on compact intervals, we obtain
$$
\int_0^1 |Q(u)|\,du = \int_1^\infty \frac{|B(y)|}{y}\,dy <\infty.
$$
implying $Q\in L^1((0,1))$. 

\begin{lemma}\label{lem:MellinQ}
For $\Re e(z)>0$,
$$
\int_0^1Q(u)u^z\,du=-\frac1{z\zeta(z+1)}
$$
and
$$
\int_0^1Q(u)\,du=-1.
$$
\end{lemma}

\begin{proof}
Substituting $y=1/u$, we obtain
\begin{equation}\label{eq:qB}
\int_0^1Q(u)u^z\,du = \int_1^\infty B(y)y^{-z-1}\,dy.
\end{equation}
Next, Abel summation, applied to
$a_n=\mu(n)/n$ and $f(t)=t^{-z}$, gives, for $X\ge1$,
$$
\sum_{n\le X}\frac{\mu(n)}{n^{z+1}} = -B(X)X^{-z} - z\int_1^X B(y)y^{-z-1}\,dy.
$$
As for $\Re e(z)>0$, $-B(X)X^{-z}\rightarrow0$ as $X\to\infty$ (since $\sum_{n=1}^{\infty}\frac{\mu(n)}{n}=0$), letting $X\to\infty$, implies
$$
\frac1{\zeta(z+1)} = -z\int_1^\infty B(y)y^{-z-1}\,dy.
$$
Together with \eqref{eq:qB} this yields
$$
\int_0^1Q(u)u^z\,du = -\frac1{z\zeta(z+1)}.
$$
for $\Re e(z)>0$ as claimed. Finally, recalling ${\rm Res}_{s=1}\zeta(s)=1$, we get $\lim_{z\to 0} \frac{-1}{z\zeta(1+z)}=-1$. Since $Q\in L^1((0,1))$, as we have seen above, dominated convergence, with $z\to0^+$, gives
$$-1=\lim_{z\to 0^+} \frac{-1}{z\zeta(1+z)}=\lim_{z\to 0^+}\int_0^1Q(u)u^z\,du
=\int_0^1Q(u)\,du,$$ 
whence the second assertion.
\end{proof}

\begin{proposition}\label{prop:kernel}
For every $z\in\C$, $\ERi(z)=-\int_0^1Q(u)e^{zu}\,du.$
\end{proposition}

\begin{proof}
For fixed $z\in\mathbb C$, the expansion $e^{zu}=\sum_{k=0}^{\infty}\frac{z^k u^k}{k!}$ converges uniformly for $u\in[0,1]$. Since $Q\in L^1((0,1))$, we may therefore integrate termwise. Using Lem.\ \ref{lem:MellinQ}, and recalling \eqref{eq:ERI}, we hence obtain
$$
-\int_0^1Q(u)e^{zu}\,du = 1+\sum_{k=1}^{\infty} \frac{z^k}{k\,k!\,\zeta(k+1)}=\ERi(z),
$$
as desired.
\end{proof}

\subsection{Non-trivial zeros as poles of a truncated Mellin transform}
Let $A>0$ be arbitrary but fixed. For $0<u<1$ put $q_A(u):=e^{-Au}Q(u)$, and define, originally for $z\in\C$ with $\Re e(z)>0$
$$
M_A(z):=\int_0^1q_A(u)u^z\,du = \int_0^1e^{-Au} \ u^z \ Q(u)\,du.
$$
We remark that since $q_A\in L^1((0,1))$, the same integral converges absolutely also on the boundary line $\Re e(z)=0$, to which we extend it. In fact, one gets much more:

\begin{lemma}\label{lem:MA}
The function $M_A$ admits the meromorphic continuation to all $z\in\C$, given by
\begin{equation}\label{eq:MA}
M_A(z)= \sum_{j=0}^{\infty} \frac{(-1)^{j+1}A^j}{j!(z+j)\zeta(z+j+1)}.
\end{equation}
If $\rho$ is a non-trivial zero of $\zeta(s)$, then $M_A(z)$ has pole at $z=\rho-1$.
\end{lemma}

\begin{proof}
For $\Re e(z)>0$, we expand $e^{-Au}$ and use Lem.\ \ref{lem:MellinQ} to obtain
$$
\begin{aligned}
M_A(z)
&= \int_0^1\sum_{j=0}^{\infty}\frac{(-Au)^j}{j!}
Q(u)u^{z}\,du\\
&= \sum_{j=0}^{\infty}\frac{(-A)^j}{j!} \int_0^1Q(u)u^{z+j}\,du\\
&= \sum_{j=0}^{\infty} \frac{(-1)^{j+1}A^j}{j!(z+j)\zeta(z+j+1)},
\end{aligned}
$$
where we just remark that exchanging summation with integration is justified by absolute convergence of $\sum_{j=0}^{\infty}\int_0^1\big | \frac{(-Au)^j}{j!}
Q(u)u^{z} \big|\,du\leq e^A \|Q\|_1<\infty$. Now, let $K\subset\C$ be a compact subset. For all sufficiently large $j$, one has, uniformly for $z\in K$,
$$
\left|\frac{(-1)^{j+1}A^j}{j!(z+j)\zeta(z+j+1)}\right|\ll_K \frac{A^j}{j!\,j},
$$
since $\zeta(z+j+1)$ converges uniformly to $1$ on $K$ as $j\to\infty$, since
$$ \sup_{z\in K} \bigl|\zeta(z+j+1)-1\bigr|\leq \sum_{n\geq 2} n^{-(j+1+\min_{z\in K}\Re e(z) )}\underset{j \to \infty}{\longrightarrow} 0.$$ Hence the tail of the series converges normally on \(K\). Since the finitely many remaining terms are meromorphic, the series \eqref{eq:MA} is indeed a meromorphic continuation of $M_A(z)$ to all $z\in\C$.\\\\
Finally, at $z=\rho-1$, where $\rho$ denotes again a non-trivial zero of $\zeta(s)$, the $j=0$ term of the series \eqref{eq:MA} has a pole of the same order as the zero
of $\zeta(s)$ at $\rho$. For every $j\ge1$, $\Re e(\rho+j)>1$, hence $\zeta(\rho+j)\neq0$. Moreover
$$
\Re e(\rho-1+j)=\Re e(\rho+j)-1>0,
$$
so $\rho-1+j\neq0$. Therefore, all terms with $j\ge1$ in \eqref{eq:MA} are all holomorphic at $z=\rho-1$, and no cancellation of the just mentioned pole of the term corresponding to $j=0$ is possible.
\end{proof}

\noindent For $Y\in\R$ let us abbreviate 
$$f_A(Y):=\Re e(\ERi(-A+iY))=-\int_0^1q_A(u)\cos(Yu)\,du.$$ 
Moreover, for $z\in\C$ with $\Re e(z)>1$, let us define
$$
H_A(z):=\int_1^\infty f_A(Y)Y^{-z}\,dY = -\int_1^\infty \int_0^1q_A(u)\cos(Yu)Y^{-z}\,du\,dY .
$$
We get

\begin{proposition}\label{prop:HA}
The truncated Mellin transform $H_A$ admits a meromorphic continuation to all $z\in\C$ by
\begin{equation}\label{eq:HAtrans}
H_A(z) = -\Gamma(1-z)\sin\left(\frac{\pi z}{2}\right)\,M_A(z-1) +
\sum_{k=0}^{\infty}
\frac{(-1)^kM_A(2k)}{(2k)!(2k+1-z)}.
\end{equation}
In particular, $H_A$ has a pole at every non-trivial zero $\rho$ of $\zeta(s)$.
\end{proposition}

\begin{proof}
We start off with the following observation: For fixed $u>0$, the integral $\int_1^\infty Y^{-z}\cos(uY)\,dY$ has an analytic continuation to an entire function in $z\in\C$ by virtue of
\begin{equation}\label{eq:KA}
\int_1^\infty Y^{-z}\cos(uY)\,dY=\Gamma(1-z)\sin\left(\frac{\pi z}{2}\right)\,u^{z-1}-\sum_{k=0}^{\infty}\frac{(-1)^k u^{2k}}{(2k)!(2k+1-z)}.
\end{equation}
Indeed, for $0<\Re e(z)<1$, the integral  converges locally uniformly by Dirichlet's test. Moreover, the standard Mellin transform gives
\begin{equation}\label{eq:KA1}
\int_0^\infty Y^{-z}\cos(uY)\,dY = \Gamma(1-z)\sin\left(\frac{\pi z}{2}\right)\,u^{z-1}.
\end{equation}
Integrating the power series of cosine on $[0,1]$ term by term, yields
\begin{equation}\label{eq:KA2}
\int_0^1Y^{-z}\cos(uY)\,dY= \sum_{k=0}^{\infty}\frac{(-1)^k u^{2k}}{(2k)!(2k+1-z)},
\end{equation}
proving \eqref{eq:KA} in the strip $0<\Re e(z)<1$. We claim that the right hand side of \eqref{eq:KA} is an entire function in $z\in\C$: In fact, since $\sin\left(\frac{\pi z}{2}\right)=0$ at positive even integers, the first term $\Gamma(1-z)\sin\left(\frac{\pi z}{2}\right)\,u^{z-1}$ given by \eqref{eq:KA1} is a meromorphic function, whose poles are at the positive odd integers. In order to analyze the second term, given by \eqref{eq:KA2}, let
$$
K\subset\mathbb C\setminus\{2m+1, m\geq 0\}
$$
be any compact subset and choose any $R>0$ such that $|w|\le R$ for all $w\in K$. For all sufficiently large $k$,
$$
|2k+1-w| \ge 2k+1-|w| \ge k \qquad (w\in K).
$$
Hence
$$
\sup_{w\in K} \left| \frac{(-1)^k u^{2k}}{(2k)!(2k+1-w)}\right|\le\frac{u^{2k}}{(2k)!\,k}.
$$
Since $\sum_{k=1}^{\infty} \frac{u^{2k}}{(2k)!\,k}<\infty,$ the series in \eqref{eq:KA2} converges absolutely and uniformly on $K$.  It therefore defines a holomorphic function on $\mathbb C\setminus\{2m+1, m\geq 0\}$, and hence a meromorphic function on $\mathbb C$, with poles (possibly) only at the positive odd integers.\\\\ 
To settle the entireness of the right hand side of \eqref{eq:KA} (and hence our intermediate claim, made at the very beginning of this proof), it remains to show that these apparent poles at the odd positive integers of both terms in \eqref{eq:KA1} and in \eqref{eq:KA2} finally cancel each other in the respective Laurent series expansion, i.e., that in the difference in \eqref{eq:KA} these singularites are removed. For this purpose we return to the original integral.  For $\Re e(z)>0$, integration by parts gives
$$\int_1^\infty Y^{-z}\cos(uY)\,dY=-\frac{\sin u}{u}+\frac{z}{u} \int_1^\infty Y^{-z-1}\sin(uY)\,dY.$$
Indeed, the boundary term at infinity vanishes since $|Y^{-z}\sin(uY)|\le Y^{-\Re e(z)}\longrightarrow 0$ as $Y\to\infty$. The integral on the right is absolutely and locally uniformly convergent in $z$ on the half-plane $\Re e(z)>0$:  More precisely, if $K_0$ is a compact subset of this half-plane, then there exists $\delta>0$ such that $\Re e(w)\ge\delta$ for all $w\in K_0$, and
$$|Y^{-w-1}\sin(uY)| \le Y^{-1-\delta}, \qquad w\in K_0,
$$
and $Y^{-1-\delta}$ is integrable on $[1,\infty)$.  Hence, for fixed $u>0$, $z\mapsto \int_1^\infty Y^{-z}\cos(uY)\,dY$ is in fact holomorphic on $\Re e(z)>0$. The meromorphic expression obtained by the right hand side of \eqref{eq:KA} above agrees with this holomorphic function on the nonempty open strip $0<\Re e(z)<1$, whence by the uniqueness of analytic continuation, the two functions must in fact agree on
$$
\{z\in\C\,| \,\Re e(z)>0\,\}\setminus\{2m+1, m\geq 0\}.
$$
But since $z\mapsto \int_1^\infty Y^{-z}\cos(uY)\,dY$ is holomorphic throughout $\Re e(z)>0$, every apparent singularity of the right hand side of \eqref{eq:KA} at
$z\in \{2m+1, m\geq 0\}$ must vanish in the difference of \eqref{eq:KA1} and \eqref{eq:KA2}. As these were the only possible poles of the right hand side of \eqref{eq:KA} it provides an analytic continuation of $z\mapsto \int_1^\infty Y^{-z}\cos(uY)\,dY$ to an entire function in $z\in\C$ as claimed.\\\\
We may now finish the proof of Prop.\ \ref{prop:HA}: For $\Re e(z)>1$, Fubini applies to the integrals defining $H_A(z)$, because $q_A\in L^1((0,1))$ and, as we have just seen, $\int_1^\infty Y^{-z}\cos(uY)\,dY$ is absolutely convergent. We get
$$
H_A(z)=-\int_0^1q_A(u)\int_1^\infty Y^{-z}\cos(uY)\,dY\,du
$$
and may hence insert the continued expression provided by \eqref{eq:KA}. This yields \eqref{eq:HAtrans} for $\Re e(z)>1$ and hence meromorphic continuation of $H_A(z)$ by Lem.\ \ref{lem:MA}, where we observe that since
$$|M_A(2k)| \leq \int_0^1 |q_A(u)|\,du =\|q_A\|_1<\infty,$$
the series 
$$\sum_{k=0}^{\infty}\frac{(-1)^k M_A(2k)}{(2k)!(2k+1-z)}$$
converges absolutely and locally uniformly on $\mathbb C\setminus\{2m+1, m\geq 0\}$. The last assertion implies furthermore that the series $\sum_{k=0}^{\infty}\frac{(-1)^k M_A(2k)}{(2k)!(2k+1-z)}$ is in fact holomorphic on the whole critical strip and hence on a neighbourhood of every non-trivial zero $\rho$ of $\zeta(s)$. Hence, in order to establish also the last claim about the poles of the right hand side of \eqref{eq:HAtrans}, it suffices to show that the remaining summand $-\Gamma(1-z)\sin\left(\frac{\pi z}{2}\right)\,M_A(z-1)$ in \eqref{eq:HAtrans} has a pole at every non-trivial zero $\rho$ of $\zeta(s)$. But this easily follows, since by Lem.\ \ref{lem:MA}, the function $M_A(z-1)$ has a pole at every non-trivial zero $\rho$ of $\zeta(s)$, whereas $\Gamma(1-\rho)\sin\left(\frac{\pi \rho}{2}\right)\neq 0$.
\end{proof}

\subsection{The auxiliary function $P_A$ and its growth}
Now let $A,C>0$ be arbitrary, but fixed, and define for $Y\in\R$, $Y>1$:
\begin{equation}\label{eq:PAdef}
P_A(Y):=\int_1^Y\log(yC^{-1})\,f_A(y)\,dy.
\end{equation}

\begin{proposition}\label{prop:Omega}
Let $\rho$ be a non-trivial zero of $\zeta(s)$. Then
$$
P_A(Y)\neq O(Y^\theta)
$$
for every $\theta<\Re e(\rho)$.
\end{proposition}

\begin{proof}
Reconsider the original definition of $H_A(z)$. Since $f_A$ is bounded, differentiation under the integral sign is justified locally uniformly for $\Re e(z)>1$: Indeed, on every compact subset of this half-plane the differentiated integrand is dominated by $O((\log Y)Y^{-1-\delta})$ for some $\delta>0$. Hence
$$
H_A'(z) = -\int_1^\infty (\log Y)f_A(Y)Y^{-z}\,dY.
$$
Therefore,
$$
\int_1^\infty \log(YC^{-1})\,f_A(Y)Y^{-z}\,dY = -H_A'(z)-(\log C)H_A(z),
$$
for $\Re e(z)>1$ and so the latter integral inherits from $H_A(z)$ a continuation to a meromorphic function $G_A(z)$ on all of $\C$. By Prop.\ \ref{prop:HA}, $H_A(z)$ has a pole at every non-trivial zero of $\zeta(s)$. If this pole has order $r\ge1$, say, then $-H_A'(z)$ has a pole of order $r+1$ (as differentiation pushes the smallest exponent in the Laurent series of $H_A(z)$ down by $1$), whereas $-(\log C)H_A$ has a pole of order at most $r$.  Hence, these poles cannot cancel, and $G_A(z)$ has a pole at every non-trivial zero of $\zeta(s)$.\\\\ 
Now suppose, for contradiction, that $P_A(Y)=O(Y^\theta)$ for some $\theta<\Re e(\rho)$ for some non-trivial zero $\rho$ of $\zeta(s)$. Since $P'_A(Y)=\log(YC^{-1})f_A(Y),$ partial integration yields, initially for $\Re e(z)>1$,
$$
\int_1^\infty \log(YC^{-1})\,f_A(Y)Y^{-z}\,dY=z\int_1^\infty P_A(Y)Y^{-z-1}\,dY.
$$
On the other hand, under the assumption $P_A(Y)=O(Y^\theta)$, the function
$$
z\longmapsto z\int_1^\infty P_A(Y)Y^{-z-1}\,dY
$$
is holomorphic in the half-plane $\Re e(z)>\theta$: Indeed, on every compact subset of this half-plane there exists $\delta>0$ such that the integrand is bounded, for large $Y$, by a constant multiple of $Y^{-1-\delta}$. The integral therefore converges absolutely and locally uniformly in $z$, yielding the claimed holomorphy for $\Re e(z)>\theta$.  But since $\int_1^\infty \log(YC^{-1})\,f_A(Y)Y^{-z}\,dY=z\int_1^\infty P_A(Y)Y^{-z-1}\,dY$, this would give a holomorphic continuation of $\int_1^\infty \log(YC^{-1})\,f_A(Y)Y^{-z}\,dY$ to $\Re e(z)>\theta$. However, as $\Re e(\rho)>\theta$, this contradicts the existence of the pole of our previous meromorphic continuation $G_A(z)$ of the aforementioned integral at $z=\rho$.
\end{proof}

\section{Contour integrals}

\subsection{A decomposition into line integrals} We shall make use of the following notion: If a real number $T>0$ is such $|\Im m(\rho)|=T$ for some zero $\rho$ of $\zeta(s)$, then we shall say that $T$ {\it is a zero ordinate for $\zeta(s)$}. Now assume that real numbers $x,\sigma\in\R$ with $x>1$ are given. If $T>0$ is not a zero ordinate, we define
$$
I_\sigma(T):=\frac1{2\pi i}\int_{\sigma-iT}^{\sigma+iT}\ERi(\log(x) s)\left(-\frac{\zeta'(s)}{\zeta(s)}\right)\,ds,
$$
whenever the integrand has no pole along the line connecting the endpoints $\sigma-iT$ and $\sigma+iT$. Also, for any fixed choice of $1<c\le2$ and $0<a<1$ let
$$H_{a,c}(T):= \frac{1}{2\pi i} \int_{c+iT}^{-a+iT} \ERi(\log(x) s)\left(-\frac{\zeta'(s)}{\zeta(s)}\right)\,ds+ \frac{1}{2\pi i}\int_{-a-iT}^{c-iT} \ERi(\log(x) s)\left(-\frac{\zeta'(s)}{\zeta(s)}\right)\,ds
$$
be the contribution of the two horizontal sides of the rectangle $\Delta_T(a,c)$ with vertices $-a\pm iT$ and $c\pm iT$ to the contour integral of $\ERi(\log(x) s)\left(-\frac{\zeta'(s)}{\zeta(s)}\right)$ along $\Delta_T(a,c)$. Observe that by the very choice of $a$ and $c$, the sides of $\Delta_T(a,c)$ do not cross any possible pole of the integrand. The reason for introducing this notation becomes clear after the following

\begin{proposition}\label{lem:residue}
Let $x\in\R$, $x>1$. Let $T>0$ be such that no zero $\rho$ of $\zeta(s)$ has imaginary part
$|\Im m(\rho)|= T$, i.e., $T$ is not a zero ordinate. Then, for any arbitrary, but fixed choice of $1<c\le2$ and $0<a<1$
$$
\ERi(\log(x))-\Sigma R_T(x)=\frac{1}{2\pi i}\oint_{\Delta_T(a,c)} \ERi(\log(x) s)\left(-\frac{\zeta'(s)}{\zeta(s)}\right)\,ds.
$$
\end{proposition}

\begin{proof}
Since $\ERi(\log(x) s)$ is entire, the only possible singularities of the integrand inside $\Delta_T(a,c)$ occur at $s=1$, arising from the pole of $\zeta(s)$ at $s=1$, and at the non-trivial zeros $\rho$ of $\zeta(s)$ with $|\Im m(\rho)|<T$, which are poles of $-\zeta'(s)/\zeta(s)$. The respective residues are
$$
\operatorname{Res}_{s=1} \ERi(\log(x)s)\left(-\frac{\zeta'(s)}{\zeta(s)}\right)=\ERi(\log(x))
$$
and, at a zero $\rho$ of multiplicity $m_\rho$,
$$
\operatorname{Res}_{s=\rho} \ERi(\log(x)s)\left(-\frac{\zeta'(s)}{\zeta(s)}\right) =-m_\rho\ERi(\log(x)\rho).
$$
The claim now follows from the residue theorem.
\end{proof}
\noindent Consequently, under the assumptions of Prop.\ \ref{lem:residue}, 
\begin{equation}\label{eq:contour}
\ERi(\log(x))-\Sigma R_T(x)= I_c(T)-I_{-a}(T)+H_{a,c}(T).
\end{equation}
In the following subsection, we will analyze this decomposition term by term.

\subsection{The horizontal contribution}
We start off analyzing the horizontal term in \eqref{eq:contour}. To this end, we shall need the following

\begin{lemma}\label{lem:decay}
For every compact interval $J\subset\mathbb R$ and every $B>0$,
$$
\ERi(\log(x)(\sigma\pm iT))
\ll_{B,J,\log(x)}(\log(T))^{-B}
$$
uniformly for $\sigma\in J$.
\end{lemma}

\begin{proof}
Using \cite{Elliott2025}, Prop.\ 4.5.2 (set $c=1$ {\it ibidem}), we obtain 
\begin{equation}\label{eq:remainder}
\ERi(z)=\sum_{n\le |z|}\frac{\mu(n)}n\Ei\left(\frac{z}{n}\right)
+O_B((\log|z|)^{-B}),
\end{equation}
which holds uniformly in any fixed vertical strip $a_1<\Re e(z)<a_2$ for $|\Im m(z)|\to\infty$. Now split the finite sum at $|z|^{1/2}$. For $n\le |z|^{1/2}$, the sectorial estimate provided by \cite{Elliott2025}, Prop.\ 4.4.1 gives
$$
\Ei\left(\frac{z}{n}\right)\ll n/|z|,
$$
so that these terms contribute an $O(|z|^{-1/2})$. For $|z|^{1/2}<n\le|z|$, put formally $g(t):=t^{-1}\Ei\left(\frac{z}{t}\right)$. Then, uniformly in our fixed
strip,
$$
g(t)\ll t^{-1},\qquad {\rm and}\qquad g'(t)= -\frac{\Ei\left(\frac{z}{t}\right)+e^{z/t}}{t^2} \ll t^{-2}.
$$
Choose a number $D>B+1$. From the classical bound for the Mertens function, see, e.g., \cite{MontgomeryVaughan}, (6.17), one deduces {\it a fortiori} that
$$
M(t):=\sum_{n\le t}\mu(n)
\ll_D t(\log(t))^{-D}.
$$
Partial summation therefore yields
$$
\begin{aligned}
\sum_{|z|^{1/2}<n\le |z|} \frac{\mu(n)}n\Ei\left(\frac{z}{n}\right) = \sum_{|z|^{1/2}<n\le |z|}\mu(n)g(n)
&=M(|z|)g(|z|) -M(|z|^{1/2})g(|z|^{1/2}) -\int_{|z|^{1/2}}^{|z|}M(t)g'(t)\,dt\\
&\ll_D (\log(|z|^{1/2}))^{-D} +(\log(|z|))^{-D} +\int_{|z|^{1/2}}^{|z|} \frac{dt}{t(\log(t))^D}\\
&\ll_D (\log(|z|^{1/2}))^{1-D},
\end{aligned}
$$
which is $O_B((\log(|z|))^{-B})$. Together with the estimate for the sum over $n\le |z|^{1/2}$ and the remainder term from \eqref{eq:remainder} we obtain, uniformly in every fixed vertical strip,
$$
\ERi(z)\ll_B(\log|z|)^{-B}
\qquad (|\Im m(z)|\to\infty).
$$
Finally take $z=\log(x)(\sigma\pm iT)$. For $\sigma\in J$, the real part $\log(x)\sigma$ ranges over the fixed compact interval $\log(x)J$, while
$$
c_1T\leq |\log(x)(\sigma\pm iT)| \leq c_2T
$$
uniformly in $\sigma$, for constants $c_1,c_2>0$ depending only on $\log(x)$ and $J$. Hence,
$$\ERi(\log(x)(\sigma\pm iT)) \ll_{B,J,\log(x)}(\log(T))^{-B},$$
as claimed.
\end{proof}

\begin{lemma}\label{lem:good}
For every sufficiently large $T$ there exists $T'\in[T,T+1]$, which is not a zero ordinate, such that
$$
\frac{\zeta'(\sigma\pm iT')}{\zeta(\sigma\pm iT')}=O((\log(T'))^2)
$$
uniformly for $-1\le\sigma\le2$.  Moreover, if $T'=T'(T)$ is any such choice made in dependence of $T$, and $1<c\le2$ and $0<a<1$ are arbitrary, but fixed, then
$$
H_{a,c}(T')=o(1)\qquad(T\to\infty).
$$
\end{lemma}

\begin{proof}
For $s=\sigma+ iT'$ the first assertion is \cite{MontgomeryVaughan}, Lem.\ 12.2, where we recall that $T'=T+O(1)$, whence the asserted bound is also $O((\log(T'))^2)$. We remark that, by complex conjugation, the same bound holds also on the lower horizontal side (i.e., for $s=\sigma- iT'$).\\\\ 
Now, on either horizontal side of our rectangle $\Delta_{T'}(a,c)$, Lem.\ \ref{lem:decay} gives, for $J:=[-a,c]$ and any fixed $B>3$,
$$
\ERi(\log(x)(\sigma\pm iT')) \ll_{B,[-a,c],\log(x)}(\log(T'))^{-B}
$$
uniformly in $-a\leq\sigma\leq c$. Multiplication by the preceding $O((\log(T'))^2)$-bound for
$$
\frac{\zeta'(\sigma\pm iT')}{\zeta(\sigma\pm iT')}
$$
and integration over segments of fixed length gives
$$
H_{a,c}(T') \ll_{B,[-a,c],\log(x)}(\log(T'))^{2-B}.
$$
Since $B>3$, this implies that $H_{a,c}(T')=o(1)$ as $T\to\infty$, for $T'=T'(T)$ as above.
\end{proof}

\subsection{The right vertical contribution and the prime-counting function $\pi_0(x)$}
We now prove the following fundamental identity for the averaged prime-counting function

\begin{proposition}\label{lem:right}
For $x\in\R$, $x>1$ and an arbitrary, but fixed $1<c\le2$, 
$$
\pi_0(x)=\displaystyle \lim_{T\to\infty}I_c(T).
$$
\end{proposition}

\begin{proof}
Writing the line $\Re e(s)=c$ by $s=c+it$, $t\in\R$, we may reparameterize
$$
I_c(T) = \frac1{2\pi}\int_{-T}^{T} \ERi(\log(x)(c+it)) \left(-\frac{\zeta'(c+it)}{\zeta(c+it)}\right)\,dt.
$$
Since $c>1$, we get for any $t\in\R$, as is well-known, cf.\ \cite{MontgomeryVaughan}, Cor.\ 1.11,
$$
-\frac{\zeta'(c+it)}{\zeta(c+it)} = \sum_{m\ge1}\frac{\Lambda(m)}{m^{c+it}}=\sum_{m\ge2}\frac{\Lambda(m)}{m^c}e^{-it\log(m)}
$$
as an absolutely convergent Dirichlet series, $\Lambda(m)$ being the Mangoldt-function, while our Prop.\ \ref{prop:kernel} gives
$$\ERi(\log(x)(c+it)) = -\int_0^1 Q(u)e^{\log(x)cu}e^{i\log(x)tu}\,du = -\int_0^1 Q_c(u)e^{i\log(x)tu}\,du,$$
with $Q_c(u):=Q(u)e^{\log(x)cu}$. Obviously, $Q_c$ lies in $L^1((0,1))$, because $\|Q_c\|_1\le e^{\log(x)c}\|Q\|_1$. Hence, for any given $T$, summation and integration may be interchanged by Fubini's theorem, since
$$
\begin{aligned}
&\frac1{2\pi}\int_{-T}^{T}\int_0^1 \sum_{m\ge2} \left| Q_c(u)\frac{\Lambda(m)}{m^c} e^{it(\log(x)u-\log(m))} \right|\,du\,dt\\
&\qquad=\frac1{2\pi}\int_{-T}^{T}\int_0^1 \sum_{m\ge2} \left| Q_c(u)\right|\frac{\Lambda(m)}{m^c} \,du\,dt\\
&\qquad= \frac{T}{\pi}\cdot \|Q_c\|_{1}\cdot \sum_{m\ge2}\frac{\Lambda(m)}{m^c}<\infty.
\end{aligned}
$$
Thus
$$
I_c(T)=-\sum_{m\ge2}\frac{\Lambda(m)}{m^c}\int_0^1Q_c(u)\left(\frac1{2\pi}\int_{-T}^{T}e^{it(\log(x)u-\log(m))}\,dt\right)du.
$$
Since
$$
\frac1{2\pi}\int_{-T}^{T}e^{itv}\,dt =\frac{\sin(Tv)}{\pi v},
$$
with the value $T/\pi$ at $v=0$, we obtain
\begin{equation}\label{eq:right-contour}
I_c(T)=-\sum_{m\ge2}\frac{\Lambda(m)}{m^c}\int_0^1Q_c(u)\frac{\sin(T(\log(x)u-\log(m)))}{\pi(\log(x)u-\log(m))}\,du.
\end{equation}
We want to determine $\lim_{T\to\infty}I_c(T)$ using \eqref{eq:right-contour}. To this end, we observe that the function $Q_c$ is integrable on $(0,1)$, and on every compact subinterval of $(0,1]$ it is of bounded variation: The only jumps of $Q$ occur at $u=1/n$, and only finitely many such jumps meet a compact set bounded away from $0$. We distinguish the case of the relative position of $x$ and $m$:\\\\
(1): If $m<x$, put $u_m=\log(m)/\log(x)\in(0,1)$. Extending $Q_c$ by zero outside $(0,1)$, Fourier's single-integral theorem, see, e.g., \cite{TitchmarshFourier}, Thm.\ 12, in \S1.14, applied at $u_m$, gives 
\begin{equation}\label{eq:right-dirichlet}
\lim_{T\to\infty}\int_0^1Q_c(u)\frac{\sin(T(\log(x)u-\log(m)))}{\pi(\log(x)u-\log(m))}\,du=\frac1{\log(x)}\,\frac{\lim_{u\to u_m^+}Q_c(u)+\lim_{u\to u_m^-}Q_c(u)}2.
\end{equation}
(2): If $m=x$, then $\log(m)=\log(x)$, and hence
$$
\int_0^1Q_c(u)\frac{\sin(T(\log(x)u-\log(m)))}{\pi(\log(x)u-\log(m))}\,du=\frac{1}{\log(x)}\int_0^1Q_c(u)\frac{\sin(T\log(x)(u-1))}{\pi(u-1)}\,du.
$$
Extend $Q_c$ by zero to the right of $u=1$.  Applying \cite{TitchmarshFourier}, Thm.\ 12 once more at the resulting jump point $u=1$, hence gives
\begin{equation}\label{eq:right-dirichlet2}
\lim_{T\to\infty}\int_0^1Q_c(u)\frac{\sin(T(\log(x)u-\log(m)))}{\pi(\log(x)u-\log(m))}\,du=\frac{\lim_{u\to 1^-}Q_c(u)}{2\log(x)}.
\end{equation}
(3): If $m>x$, then $\log(m)-\log(x)>0$, and hence
\begin{equation}\label{eq:huiui}
\int_0^1\left| \frac{Q_c(u)}{\pi(\log(x)u-\log(m))}\right|\,du \leq \int_0^1\frac{|Q_c(u)|}{\pi(\log(m)-\log(x))}\,du=\frac{\|Q_c\|_1}{\pi(\log(m)-\log(x))}<\infty,
\end{equation}
implying that $h_m(u):=\frac{Q_c(u)}{\pi(\log(x)u-\log(m))}$ is in $L^1((0,1))$. Moreover,
$$
\begin{aligned}
\int_0^1Q_c(u)\frac{\sin(T(\log(x)u-\log(m)))}{\pi(\log(x)u-\log(m))}\,du&=\Im m\left(e^{-iT\log(m)}\int_0^1 h_m(u)e^{iT\log(x)u}\,du\right).
\end{aligned}
$$
Since $\log(x)>0$, the Riemann--Lebesgue theorem, cf.\  \cite{TitchmarshFourier}, Thm.\ 1 in \S1.8, implies that the last integral tends to $0$ as $T\to\infty$ for every fixed $m>x$. Moreover, by what we have just observed in \eqref{eq:huiui}, 
$$\left|\int_0^1Q_c(u) \frac{\sin(T(\log(x)u-\log(m)))}{\pi(\log(x)u-\log(m))}\,du\right|\le \frac{\|Q_c\|_1}{\pi(\log(m)-\log(x))},
$$
and hence
$$
\frac{\Lambda(m)}{m^c}\left|\int_0^1Q_c(u) \frac{\sin(T(\log(x)u-\log(m)))}{\pi(\log(x)u-\log(m))}\,du\right|\le\frac{\|Q_c\|_1}{\pi}\frac{\Lambda(m)}{m^c(\log(m)-\log(x))}.
$$
The majorant is summable: Indeed, for all sufficiently large $m\gg x$, $\log(m)-\log(x)\ge\frac12\log(m)$, and since also $\Lambda(m)\le\log(m)$,
$$
\frac{\Lambda(m)}{m^c(\log(m)-\log(x))}\le \frac{2}{m^c}.
$$
Because $c>1$, $\sum_{m\gg x}\frac{2}{m^c}$ converges -- the finitely many remaining
terms cause no difficulty. Therefore, dominated convergence yields
$$
\lim_{T\to\infty}\left(\sum_{m>x}\frac{\Lambda(m)}{m^c} \int_0^1Q_c(u) \frac{\sin(T(\log(x)u-\log(m)))}{\pi(\log(x)u-\log(m))}\,du\right)=
$$
$$
=\sum_{m>x}\frac{\Lambda(m)}{m^c} \lim_{T\to\infty} \left(\int_0^1Q_c(u) \frac{\sin(T(\log(x)u-\log(m)))}{\pi(\log(x)u-\log(m))}\,du\right)$$
$$=0.
$$
It follows that the contribution of those summands of the right hand side of \eqref{eq:right-contour}, which are indexed by $m>x$, to the limit $\lim_{T\to\infty}I_c(T)$ vanishes. \\\\
It hence remains to evaluate the contribution of the summands with $m\leq x$. To this end we use \eqref{eq:right-dirichlet} and \eqref{eq:right-dirichlet2}: In order to treat the left- and right-sided limits therein most efficiently, we define for $y\ge1$ the auxiliary function 
\begin{equation}\label{eq:A0}
A_0(y):=\sum_{n<y}\frac{\mu(n)}n+\frac12\,\mathbf 1_{\mathbb N}\frac{\mu(y)}y,
\end{equation}
where $\mathbf 1_{\mathbb N}$ denotes the characteristic function of $\N=\{1,2,3,...\}\subset\R$. We now relate the averaged one-sided values of $B$ to $A_0$.  We
claim that, for every $y\ge1$,
\begin{equation}\label{eq:B-average-A0}
-\frac{\lim_{y\to y^-}B(y)+\lim_{y\to y^+}B(y)}2=A_0(y).
\end{equation}
Indeed, if $y\notin\mathbb N$, then $B$ is continuous at $y$, and using $\sum_{n=1}^{\infty}\frac{\mu(n)}n=0$, we obtain
$$
-B(y)=-\sum_{n>y}\frac{\mu(n)}n=\sum_{n<y}\frac{\mu(n)}n=A_0(y).
$$
If $y=N\in\mathbb N$, then
$$
\lim_{y\to N^-}B(y)=\frac{\mu(N)}N+\sum_{n>N}\frac{\mu(n)}n,\qquad\lim_{y\to N^+}B(y)=\sum_{n>N}\frac{\mu(n)}n.
$$
Therefore
$$
\begin{aligned}
-\frac{\lim_{y\to N^-}B(y)+\lim_{y\to N^+}B(y)}2
&=-\sum_{n>N}\frac{\mu(n)}n-\frac12\frac{\mu(N)}N\\
&=\sum_{n<N}\frac{\mu(n)}n+\frac12\frac{\mu(N)}N\\
&=A_0(N),
\end{aligned}
$$
which proves \eqref{eq:B-average-A0}.\\\\
Having made this intermediate observation, we now resume our discussion of the contribution of the summands indexed by $m<x$. Using \eqref{eq:right-dirichlet}, and reinserting into the definition, namely
$$
Q_c(u)=e^{\log(x)cu}Q(u)=e^{\log(x)cu}u^{-1} B(u^{-1})=e^{\log(x)cu} u^{-1}\sum_{n>u^{-1}}\frac{\mu(n)}{n},
$$
it follows that the limiting contribution of a summand index by $m<x$ to the limit $\lim_{T\to\infty}I_c(T)$ is
$$
-\frac{\Lambda(m)}{\log(m)}\,\frac{\lim_{y\to (u^{-1}_m)^+}B(y)+\lim_{y\to (u^{-1}_m)^-} B(y)}2 =\frac{\Lambda(m)}{\log(m)}\,A_0\!\left(\frac{\log(x)}{\log(m)}\right).
$$
If $m=x$ is an integer, then $u_m=1$. We have 
$$\lim_{u\to 1^-}Q_c(u)=\lim_{y\to 1^+}e^{\log(m)cy}B(y)=-m^c.$$
Therefore, \eqref{eq:right-dirichlet2} shows that the contribution of $m=x$ to the limit of $I_c(T)$ is
$$
-\frac{\Lambda(m)}{m^c}\left(-\frac{m^c}{2\log(m)}\right)=\frac{\Lambda(m)}{2\log m}=\frac{\Lambda(m)}{\log(m)}
A_0\!\left(\frac{\log(x)}{\log(m)}\right),
$$
where the last equation follows from the fact that $A_0(1)=\frac12$.\\\\
Altogether, we therefore obtain the unified formula
\begin{equation*}\label{eq:right-limit-A0}
\lim_{T\to\infty}I_c(T)=\sum_{m\le x}\frac{\Lambda(m)}{\log(m)}\,A_0\!\left(\frac{\log(x)}{\log(m)}\right).
\end{equation*}
Expanding the definition of $A_0(y)$, we obtain
\begin{equation}\label{eq:right-expanded}
\lim_{T\to\infty}I_c(T)=\sum_{\substack{m\ge2,\ n\ge1\\ m^n<x}}\frac{\Lambda(m)}{\log m}\frac{\mu(n)}n+\frac12\sum_{\substack{m\ge2,\ n\ge1\\ m^n=x}}\frac{\Lambda(m)}{\log m}\frac{\mu(n)}n.
\end{equation}
Since $\Lambda(m)=0$ unless $m$ is a prime power, we may assume $m=p^k$, with $p$ prime and $k\ge1$ in the above sums. Then, for any $n\geq 1$, 
$$
\frac{\Lambda(p^k)}{\log(p^k)} \frac{\mu(n)}n= \frac{\log p}{k\log p} \frac{\mu(n)}n= \frac{\mu(n)}{kn}.
$$
If we put in turn $r=kn$, then for fixed $p$ and $r$, the total contribution of the number $p^r=m^n$ to each of the two sums in \eqref{eq:right-expanded} is
$$
\frac1r\sum_{n\mid r}\mu(n)=
\begin{cases}
1,&r=1,\\
0,&r>1,
\end{cases}
$$
by the classical M\"obius identity $\sum_{n\mid r}\mu(n)=\delta_{r,1}$. Consequently, the contribution of all proper prime powers $p^r$, $r>1$, to \eqref{eq:right-expanded} is trivial. Therefore, we are left with
$$
\lim_{T\to\infty}I_c(T)=\sum_{p<x}1 +
\begin{cases}
1/2& \textrm{if $x$ is prime},\\
0 & \textrm{else},
\end{cases}
$$
which is nothing else than the averaged prime-counting function $\pi_0(x)$. This proves the result.
\end{proof}

\subsection{The left vertical contribution and the functional equation of $\zeta(s)$}
We now use the functional equation of the $\zeta$-function. Its logarithmic derivative gives
$$
-\frac{\zeta'(s)}{\zeta(s)} = \psi(1-s)-\log(2\pi) -\frac{\pi}{2}\cot\left(\frac{\pi s}{2}\right) +\frac{\zeta'(1-s)}{\zeta(1-s)}.
$$
Here $\psi(s)=\Gamma'(s)/\Gamma(s)$ denotes the digamma function. By Stirling's expansion for the logarithmic derivative of the gamma function, see, e.g., \cite{MontgomeryVaughan}, (C.20), one has uniformly in sectors bounded away from the negative real axis, i.e., $|\arg(z)|\leq \pi-\delta$ for $\delta>0$,
$$
\psi(z)=\log z-\frac{1}{2z}+O(|z|^{-2}).
$$
Thus, for fixed $0<a<1$ and $t\to\infty$,
$$
\psi(1+a-it) = \log(t)-\frac{\pi i}{2} +\frac{i(a+\tfrac12)}{t} +O_a(t^{-2}).
$$
Moreover,
$$
\cot\left(\frac{\pi(-a+it)}2\right) =i\frac{e^{\pi i(-a+it)}+1}{e^{\pi i(-a+it)}-1} = -i+O_a(e^{-\pi t}).
$$
Substitution into the logarithmic derivative of the functional equation therefore gives
\begin{equation}\label{stirling}
-\frac{\zeta'(-a+it)}{\zeta(-a+it)} = \log\left(\frac{t}{2\pi}\right) +\frac{i(a+\tfrac12)}t +\frac{\zeta'(1+a-it)}{\zeta(1+a-it)} +O_a(t^{-2}).
\end{equation}
Next recall the function $P_A(Y)$ from \eqref{eq:PAdef}. For our given, arbitrary, but fixed real numbers $x>1$ and $0<a<1$, we now specify $A:=a\log(x)>0$ and $C:=2\pi \log(x)>0$ in its definition, i.e., we consider for $Y>1$
$$
P_A(Y) = \int_1^Y \log(y(2\pi \log(x))^{-1})\,
\Re e(\ERi(-a\log(x)+iy))\,dy,
$$

\begin{proposition}\label{prop:left}
There is a constant $C_{a,x}$, depending on $a$ and $x$, but independent of $T$, such that
$$
I_{-a}(T) = \frac{1}{\pi \log(x)}P_A(\log(x)T) +  C_{a,x} + o(1)
$$
as $T\to\infty$.
\end{proposition}

\begin{proof}
Reparameterizing the integral defining $I_{-a}(T)$, we get
$$
I_{-a}(T) = \frac1{2\pi}\int_{-T}^T\ERi(\log(x)(-a+it))\left(-\frac{\zeta'(-a+it)}{\zeta(-a+it)}\right)dt.
$$
Its integrand $g(t)$ satisfies $g(-t)=\overline{g(t)}$, whence dividing the integration as $\int_{-T}^T=\int_{0}^T+\int_{-T}^0$, and flipping signs/boundaries in the latter, one gets
$$
I_{-a}(T) = \frac1\pi \Re e\left(\int_0^T \ERi(-A+i\log(x)t) \left(-\frac{\zeta'(-a+it)}{\zeta(-a+it)}\right)\,dt\right)
$$
Choose some $t_0\ge2$, sufficiently large such that $\log(x)t_0\geq 1$ and such that \eqref{stirling} holds uniformly for $t\ge t_0$. Since $a$ and $x$ are fixed, the contribution of the compact interval $0\le t\le t_0$ to $I_{-a}(T)$ is a finite constant depending only on $a$ and $x$, and will henceforth be absorbed into what will finally be $C_{a,x}$. Thus, inserting \eqref{stirling}, it remains to analyze
\begin{equation}\label{eq:final1} 
\frac1\pi \Re e\left(\int_{t_0}^T \ERi(-A+i\log(x)t) \left(\log\left(\frac{t}{2\pi}\right) +\frac{i(a+\tfrac12)}t +\frac{\zeta'(1+a-it)}{\zeta(1+a-it)} +O_a(t^{-2})\right)\right)dt
\end{equation}
Its summand attached to the $O_a(t^{-2})$-term is absolutely integrable, because $\ERi(-A+i\log(x)t)$ is bounded as a function in $t\in\R$ by Prop.\ \ref{prop:kernel}. It hence contributes a constant -- which will again be absorbed into $C_{a,x}$ -- and an $o(1)$-term to \eqref{eq:final1}, or, equivalently, to $I_{-a}(T)$.\\\\ 
We next consider the summand of \eqref{eq:final1} attached to $\frac{i(a+\tfrac12)}{t}$. Its contribution to \eqref{eq:final1} is
$$
-\frac{a+\tfrac12}{\pi} \int_{t_0}^T \frac{\Im m(\ERi(-A+i\log(x)t))}{t}\,dt,
$$
since $\Re e(iz)=-\Im m(z)$. We claim that this expression converges as $T\to\infty$. Indeed, by our kernel representation, cf.\ again Prop.\ \ref{prop:kernel}, 
$$
\Im m(\ERi(-A+i\log(x)t)) = -\int_0^1q_A(u)\sin(\log(x)tu)\,du.
$$
For fixed $T$, Fubini's theorem therefore gives
$$
\begin{aligned}
\int_{t_0}^T\frac{\Im m(\ERi(-A+i\log(x)t))}{t}\,dt
&=-\int_0^1q_A(u)\left(\int_{t_0}^T\frac{\sin(\log(x)tu)}t\,dt\right)du\\
&=-\int_0^1q_A(u)\bigl(\operatorname{Si}(T\log(x)u)-\operatorname{Si}(t_0\log(x)u)\bigr)\,du,
\end{aligned}
$$
where $\operatorname{Si}(y):=\int_0^y\frac{\sin v}{v}\,dv$ denotes the sine-integral. The latter is bounded as a function on $y\in [0,\infty)$ with $\lim_{y\to\infty}\operatorname{Si}(y)=\tfrac{\pi}{2}.$ Hence the expression $\operatorname{Si}(T\log(x)u)-\operatorname{Si}(t_0\log(x)u)$ is uniformly bounded in $T\ge t_0$ and $0<u<1$, and, for every fixed $u>0$, tends to $\frac{\pi}{2}-\operatorname{Si}(t_0\log(x)u)$ as $T\to\infty$. Since $q_A\in L^1((0,1))$, dominated convergence now yields
$$
\lim_{T\to\infty}\int_{t_0}^T\frac{\Im m(\ERi(-A+i\log(x)t))}{t}\,dt = -\int_0^1q_A(u) \left(\frac{\pi}{2}-\operatorname{Si}(t_0\log(x)u)\right)du<\infty.
$$
Thus, the contribution of the term attached to $\frac{i(a+\tfrac12)}{t}$ in \eqref{eq:final1} to $I_{-a}(T)$ is a constant -- which we may once more absorb into our final choice of $C_{a,x}$ -- together with an $o(1)$-term.\\\\
We next treat the term in \eqref{eq:final1}, which is attached to the absolutely convergent (recall that $0<a$) Dirichlet series
$$
\frac{\zeta'(1+a-it)}{\zeta(1+a-it)} = -\sum_{m\ge2}\frac{\Lambda(m)}{m^{1+a}}e^{it\log(m)}.
$$
Since $\ERi(-A+i\log(x)t)=-\int_0^1q_A(u)e^{i\log(x)tu}\,du$, cf.\ Prop.\ \ref{prop:kernel}, the contribution of this Dirichlet-series term to \eqref{eq:final1} equals
\begin{equation}\label{eq:left-dirichlet}
\frac1\pi\Re e\left(\sum_{m\ge2}b_m\int_0^1q_A(u)\frac{e^{iT\lambda_m(u)}-e^{it_0\lambda_m(u)}}{i\lambda_m(u)}\,du\right),
\end{equation}
where we have abbreviated $b_m:=\Lambda(m)m^{-1-a}$ and $\lambda_m(u):=\log(x)u+\log(m)$. We observe that the just performed interchange of integration and summation is justified by the absolute convergence of $\sum_{m\geq 2} b_m$ and by the fact that $q_A\in L^1((0,1))$. Now, the second exponential showing up in \eqref{eq:left-dirichlet} is independent of $T$ and hence contributes a fixed constant. For the first one, observe that $\lambda_m(u)\geq\log(m)\geq\log2>0$, independent of $u$, and hence for every fixed $m\geq 2$
$$
\lim_{T\to\infty}\int_0^1q_A(u)\frac{e^{iT\lambda_m(u)}}{i\lambda_m(u)}\,du=0
$$
by the Riemann--Lebesgue theorem, see again \cite{TitchmarshFourier}, Thm.\ 1 in \S1.8. Moreover, for each $T\geq t_0$,
$$
\left|\int_0^1q_A(u)\frac{e^{iT\lambda_m(u)}}{\lambda_m(u)}\,du\right|\le\frac{\|q_A\|_1}{\log(m)}.
$$
independent of $T$. Therefore, for each $T\geq t_0$,
$$\left|\sum_{m\ge2}b_m\int_0^1q_A(u)\frac{e^{iT\lambda_m(u)}}{i\lambda_m(u)}\,du\right| \leq \|q_A\|_1\sum_{m\ge2}\frac{\Lambda(m)}{m^{1+a}\log(m)}\leq \|q_A\|_1\sum_{m\ge2}\frac{1}{m^{1+a}}<\infty$$
as $a>0$. As a consequence, dominated convergence shows that 
$$\lim_{T\to\infty}\left(\sum_{m\ge2}b_m\int_0^1q_A(u)\frac{e^{iT\lambda_m(u)}}{i\lambda_m(u)}\,du\right) = \sum_{m\ge2}b_m \lim_{T\to\infty}\left(\int_0^1q_A(u)
\frac{e^{iT\lambda_m(u)}}{i\lambda_m(u)}\,du\right)= 0$$
It follows that the term corresponding to $\frac{\zeta'(1+a-it)}{\zeta(1+a-it)} $ in \eqref{eq:final1} contributes a constant to $I_{-a}(T)$ -- which will again be absorbed into the final $C_{a,x}$ -- together with an $o(1)$-term.\\\\
We have therefore proved that every term furnished by the functional equation except the logarithmic term contributes a constant, only depending on our fixed $0<a<1$ and $x>1$, together with an $o(1)$-term to $I_{-a}(T)$ as $T\to\infty$ . To analyze the last remaining ingredient, i.e., to determine the contribution of the summand of \eqref{eq:final1} attached to $\log\left(\frac{t}{2\pi}\right)$ to $I_{-a}(T)$, we observe that changing variables by $y=\log(x)t$ gives
$$
\frac1\pi\int_{t_0}^T\log\left(\frac{t}{2\pi}\right)\,f_A(\log(x)t)\,dt = \frac1{\pi \log(x)}\int_{\log(x)t_0}^{\log(x)T} \log\left(\frac{y}{2\pi \log(x)}\right)\,f_A(y)\,dy,
$$
where we recall that $f_A(y)=\Re e(\ERi(-A+iy))$ by definition. Hence, the contribution of the term corresponding to $\log\left(\frac{t}{2\pi}\right)$ to \eqref{eq:final1} equals
$$
\begin{aligned}
\frac1\pi \Re e\left(\int_{t_0}^T \ERi(-A+i\log(x)t) \log\left(\frac{t}{2\pi}\right) dt\right) 
&=\frac1\pi \int_{t_0}^T \Re e\left(\ERi(-A+i\log(x)t)\right) \log\left(\frac{t}{2\pi}\right) \\
&=\frac1{\pi \log(x)}\int_{\log(x)t_0}^{\log(x)T} \log\left(\frac{y}{2\pi \log(x)}\right)\,f_A(y)\,dy,\\
&=\frac1{\pi \log(x)} P_A(\log(x)T) + C'_{a,x},
\end{aligned}
$$
for $C'_{a,x}:=\frac{-1}{\pi\log(x)}\int_{1}^{\log(x)t_0}\log\left(\frac{y}{2\pi \log(x)}\right)\,f_A(y)\,dy$. Adding $C'_{a,x}$ to the constants obtained so far from the other summands in the functional equation only alters the final choice of our constant $C_{a,x}$. The proof is complete.
\end{proof}

\section{Proof of the main theorem}\label{sect:proof}
\noindent We are now ready to give the proof of our main result. Recall the partial sums, defined for real numbers $x>1$ and $T>0$ as
$$
\Sigma R_T(x)= \sum_{\substack{\zeta(\rho)=0\\0<|\Im m(\rho)|\le T}} m_\rho\ERi(\rho\log x),
$$
where $m_\rho$ denotes the finite multiplicity of the zero $\rho$ of $\zeta(s)$ and also recall ``Riemann's constant'' 
$$\Theta:=\sup\{\Re e(\rho):\ \zeta(\rho)=0,\ 0<\Re e(\rho)<1\}.$$

\begin{proof}[Proof of the Main Theorem:]
Let $x>1$ be arbitrary, but fixed and assume for contradiction that there is a $\theta\in\R$ with $0<\theta<\Theta$ such that
\begin{equation}\label{eq:S-assumption}
\Sigma R_T(x)=O(T^\theta).
\end{equation}
By the definition of $\Theta$, there exists a non-trivial zero $\rho$ of $\zeta(s)$ with $\beta:=\Re e(\rho)>\theta$. Applying Lem.\ \ref{lem:good} for sufficiently large $T$, we may choose $T'\in[T,T+1]$ such that $T'$ is not the ordinate of a zero and get, for any chosen $0<a<1$ and $1<c\leq 2$,
$$H_{a,c}(T')=o(1) \qquad(T\to\infty),$$
Along such ordinates $T'$, Prop.\ \ref{lem:residue} applies as well. Hence, we get, see \eqref{eq:contour},
\begin{equation}\label{eq:transfer-good}
\Sigma R_{T'}(x) = \ERi(\log(x))-I_c(T')+I_{-a}(T')-o(1),
\end{equation}
along the chosen ordinates $T'=T'(T)$, as $T\to\infty$. Invoking Prop.\ \ref{lem:right} and Prop.\ \ref{prop:left}, we hence obtain
$$ \Sigma R_{T'}(x) = \ERi(\log(x))-(\pi_0(x)+o(1))+\left(\frac1{\pi \log(x)}P_A(\log(x)T')+C_{a,x}+o(1)\right)-o(1),$$
as $T'\to\infty$ (i.e., $T\to\infty$). Subsuming all summands, which are independent of $T'$ into one constant $D_{a,x}$, this simplifies to
$$\Sigma R_{T'}(x) = \frac1{\pi \log(x)}P_A(\log(x)T')+D_{a,x}+o(1).$$
By the assumed bound \eqref{eq:S-assumption}, this implies that
\begin{equation}\label{eq:P-good-bound}
P_A(\log(x)T')=O(T'^\theta).
\end{equation}
Since $T'\in[T,T+1]$, we have $T'=T+O(1)$, and hence
$$
P_A(\log(x)T')=O(T^\theta).
$$
Moreover, Prop.\ \ref{prop:kernel} implies that $f_A$ is bounded, whence
$$
P_A'(t)=\log\left(\frac{t}{2\pi\log(x)}\right)f_A(t)=O_{a,x}(\log(t)).
$$
It follows that
$$
\begin{aligned}
|P_A(\log(x)T')-P_A(\log(x)T)| &=\left|\int_{\log(x)T}^{\log(x)T'} P'_A(t)\,dt\right|\\
&\leq |\log(x)T'-\log(x)T |  \cdot O(\log(\log(x)T')) \\
&=O_{a,x}(\log(T'))\\
&=O_{a,x}(\log(T))
\end{aligned}
$$
where the second last equation is a consequence of $|\log(x)T'-\log(x)T |\leq \log(x)$ and the last equation uses again $T'=T+O(1)$. Therefore,
$$
P_A(\log(x)T)=O(T^\theta)-O_{a,x}(\log(T))=O(T^\theta),
$$
since $\theta>0$. As $\log(x)>0$ is fixed, replacing $T$ by
$T/\log(x)$ yields
$$
P_A(T)=O(T^\theta).
$$
But this contradicts Prop.\ \ref{prop:Omega}, since the nontrivial zero $\rho$ chosen above satisfies $\Re e(\rho)=\beta>\theta$. Hence, $\Sigma R_T(x)\neq O(T^\theta)$ for every $0<\theta<\Theta$, and so, clearly for all $\theta<\Theta$.
\end{proof}

\begin{remark}\label{rem:exact-power-scale}
It is not too hard to see that the exponent $\Theta$ occurring in the Main Theorem is best possible: Indeed, for every fixed $x>1$ and every
$\varepsilon>0$, one has
\begin{equation}\label{eq:upper-bound-Theta}
   \Sigma R_T(x)=O_{x,\varepsilon}\bigl(T^{\Theta+\varepsilon}\bigr).
\end{equation}
We briefly indicate the argument. If $\Theta=1$, the required estimate follows immediately from the boundedness of $f_A$. In fact, since $q_A\in L^1((0,1))$, we get $ |f_A(y)|\leq \|q_A\|_{L^1(0,1)}$, and hence
$$
|P_A(Y)| \leq \|q_A\|_{L^1((0,1))}\int_1^Y |\log(yC^{-1})|\,dy \ll_{A,C}Y\log(2Y)\ll_{A,C,\varepsilon}Y^{1+\varepsilon}.
$$
Suppose now that $\Theta<1$, and choose $\eta>0$ such that $\Theta+\eta<1$ and $2\eta<\varepsilon$. We recall $M(t)=\sum_{n\leq t}\mu(n)$ from the proof of Lem.\ \ref{lem:decay}. Let $1/2\leq\vartheta<1$ and suppose that $\zeta(s)\neq0$ for $\Re e(s)>\vartheta$. Then, it is well known that $M(t)\ll_{\vartheta,\delta}t^{\vartheta+\delta}$ for every $\delta>0$, which may be seen as in \cite{MontgomeryVaughan}, \S13.2, especially Thm.\ 13.24, whose proof applies to an arbitrary zero-free boundary $\vartheta\in[1/2,1)$. Since $\zeta(s)$ has no zeros in $\Re e(s)>\Theta$, we hence obtain 
$$
 M(t)\ll_\eta t^{\Theta+\eta}.
$$
Partial summation therfore gives
$$
  B(t)=-\frac{M(t)}t+\int_t^\infty\frac{M(y)}{y^2}\,dy\ll_\eta t^{\Theta-1+\eta}.
$$
Moreover, $q_A(u)\ll_\eta u^{-\Theta-\eta}$. Set now 
$$K_Y(u):=\int_1^Y\log(yC^{-1})\cos(yu)\,dy.$$
Then, $P_A(Y)=-\int_0^1q_A(u)K_Y(u)\,du.$ Trivially, $K_Y(u)\ll_C Y\log(2Y)$, whereas integration by parts yields
$$
K_Y(u)=\left[ \log(yC^{-1})\frac{\sin(yu)}{u}\right]_{y=1}^{y=Y}-\frac1u\int_1^Y\frac{\sin(yu)}{y}\,dy \ll_C\frac{\log(2Y)}{u}.
$$
So, splitting the integral defining $P_A(Y)$ at $u=Y^{-1}$ and using the above bounds finally gives
$$
 |P_A(Y)|\ll_{A,C,\eta} Y\log(2Y)\int_0^{Y^{-1}}u^{-(\Theta+\eta)}\,du+ \log(2Y)\int_{Y^{-1}}^1u^{-(\Theta+\eta)-1}\,du\ll_{A,C,\eta}
 Y^{\Theta+\eta}\log(2Y)
 \ll_{A,C,\varepsilon}Y^{\Theta+\varepsilon},
$$
also when $\Theta<1$. As in the proof of the Main Theorem, one now concludes that this in fact implies
$$
   \Sigma R_T(x)=O_{x,\varepsilon}\bigl(T^{\Theta+\varepsilon}\bigr), 
$$
i.e., \eqref{eq:upper-bound-Theta}. We also note that combining \eqref{eq:upper-bound-Theta} with the Main Theorem yields
$$
 \limsup_{T\to\infty}\frac{\log\bigl(2+|\Sigma R_T(x)|\bigr)}{\log T} = \Theta
$$
and hence a criterion to formulate the Riemann Hypothesis.
\end{remark}

\section{Complementa}\label{sect:41}
\subsection{Remarks on the contribution of the trivial zeros} 
Let us, for completeness, also consider the second summand of \eqref{eq:Rdef2}, which involves the trivial zeros of $\zeta(s)$. As mentioned in \cite{Elliott2025}, one sometimes encounters the claimed formula \eqref{eq:trivRs}, 
$$\sum_{k\geq 1} \ERi(-2k\log(x))=\frac{1}{\log(x)} - \frac{1}{\pi} \arctan{\frac{\pi}{\log(x)}},$$ 
which seems to be a derivation of another problematic interchange of summation in two conditionally convergent series (and a certain ``regularized reinterpretation'' of $\sum_{n\geq 1}\mu(n)$ as $\zeta(0)^{-1}=-2$). See \cite{Elliott2025}, pp.\ 191--192 for details.\\\\ 
However, the left hand side $\sum_{k\geq 1} R(x^{-2k})=\sum_{k\geq 1}\ERi(-2k\log(x))$ of \eqref{eq:trivRs} does not converge, for reasons quite similar, but simpler than the ones entering our main theorem above. We summarize this in the following result: For real numbers $x>1$ and $T>1$ put 
$$\Sigma R^{triv}_T(x):=\sum_{1\leq k\leq T}\ERi(-2k\log(x)).$$
We obtain

\begin{theorem}\label{prop:trivial-zeros}
Let $x>1$ be arbitrary, but fixed. Then, for every $\theta<\Theta$,  the partial sums $\Sigma R^{triv}_T(x)$ are not $O(T^\theta)$ as $T\to\infty$. In particular,
$\limsup_{N\to\infty}\left|\sum_{1\leq k\leq N} R(x^{-2k})\right|=\infty$, and hence the series $\sum_{k\geq1}R(x^{-2k})$ does not converge.
\end{theorem}

\begin{proof}
Fix $x>1$ once and for all and put for simplicity $a_k:=\ERi(-2k\log(x))$, for $k\geq 1$. By Prop.\ \ref{prop:kernel}, $|a_k|\leq \|Q\|_1$, and, by dominated convergence, $a_k\to0$ as $k\to\infty$. For $\Re e(s)>1$, we consider the absolutely convergent Dirichlet series
$$
D_x(s):=\sum_{k\geq1}\frac{a_k}{k^s}.
$$
In a first step, we want to show that $D_x(s)$  admits a meromorphic continuation to $\Omega:=\mathbb C\setminus\{1,2,3,\ldots\}$. We observe that we have,
$$
\sum_{k\geq1}\frac1{k^{\Re e(s)}} \int_0^1|Q(u)|e^{-2\log(x)ku}\,du \leq \|Q\|_1\sum_{k\geq1}\frac1{k^{\Re e(s)}} <\infty,
$$
hence Fubini's theorem applies and gives
\begin{equation}\label{eq:DL-polylog}
D_x(s) = -\int_0^1Q(u)\operatorname{Li}_s(e^{-2\log(x)u})\,du,
\qquad \Re e(s)>1,
\end{equation}
where we wrote as usual $\operatorname{Li}_s(q):=\sum_{k\geq1}\frac{q^k}{k^s}$, $0<q<1$, for the polylogarithm. Now, choose $0<\delta<1$ so small that $v_0:=2\log(x)\delta<2\pi$. For $s\in\Omega$ and $0\leq v\leq v_0$, define
$$
G(s,v):= \sum_{j=0}^\infty \frac{(-1)^j\zeta(s-j)}{j!}v^j.
$$
This series converges locally uniformly in $s\in\Omega$ and uniformly for $0\leq v\leq v_0$: Indeed, let $K\subset\Omega$ be compact. The functional equation of the Riemann zeta function gives
$$
\zeta(s-j) = 2^{s-j}\pi^{s-j-1} \sin\left(\frac{\pi(s-j)}2\right) \Gamma(j+1-s)\zeta(j+1-s).
$$
Now, for all sufficiently large $j$ one has uniformly for $s\in K$, $\zeta(j+1-s)\ll_K1$, while Stirling's formula gives
$$
\frac{\Gamma(j+1-s)}{\Gamma(j+1)}
\ll_K j^{C_K}
$$
for a suitable constant $C_K>0$. Hence, for sufficiently large $j$, 
$$
\left| \zeta(s-j)\frac{(-v)^j}{j!} \right| \ll_K j^{C_K}\left(\frac{v_0}{2\pi}\right)^j, \qquad s\in K,\quad 0\leq v\leq v_0.
$$
Since $v_0<2\pi$, the majorant on the right is summable. Thus, 
the series defining $G(s,v)$ converges absolutely and uniformly for $(s,v)\in K\times[0,v_0]$ -- the finitely many remaining summands cause no difficulty. Consequently, $G(\,\cdot\,,v)$ is holomorphic on $\Omega$ for every $0\leq v\leq v_0$, and $\sup_{\substack{s\in K\\0\leq v\leq v_0}}|G(s,v)|<\infty$ for every compact $K\subset\Omega$. We now use the classical Jonqui\`ere expansion
$$
\operatorname{Li}_s(e^{-v}) = \Gamma(1-s)v^{s-1}+G(s,v), \qquad s\in\Omega,\quad 0<v<2\pi,
$$
see \cite{FornbergKolbig}, (26). For $1<\Re e(s)<2$, we may insert this into \eqref{eq:DL-polylog}. Indeed, $|u^{s-1}|\leq1$ on $(0,1)$, and $Q\in L^1((0,1))$. We obtain
$$
\begin{aligned}
D_x(s) ={}& -\Gamma(1-s)(2\log(x))^{s-1}  \int_0^\delta Q(u)u^{s-1}\,du-\int_0^\delta Q(u)G(s,2\log(x)u)\,du -\int_\delta^1Q(u)\operatorname{Li}_s(e^{-2\log(x)u})\,du.
\end{aligned}
$$
On the other hand, Lem.\ \ref{lem:MellinQ} gives, for $\Re e(s)>1$,
$$
\begin{aligned}
\int_0^\delta Q(u)u^{s-1}\,du &=\int_0^1Q(u)u^{s-1}\,du-\int_\delta^1Q(u)u^{s-1}\,du\\
&=-\frac{1}{(s-1)\zeta(s)}-\int_\delta^1Q(u)u^{s-1}\,du.
\end{aligned}
$$
Consequently,
\begin{equation}\label{eq:DL-cont}
D_x(s) = \frac{\Gamma(1-s)(2\log(x))^{s-1}}{(s-1)\zeta(s)} + H_{x,\delta}(s), \qquad 1<\Re e(s)<2,
\end{equation}
where
$$
\begin{aligned}
H_{x,\delta}(s) :={}&
\Gamma(1-s)(2\log(x))^{s-1}  \int_\delta^1Q(u)u^{s-1}\,du-\int_0^\delta Q(u)G(s,2\log(x)u)\,du-\int_\delta^1Q(u)\operatorname{Li}_s(e^{-2\log(x)u})\,du.
\end{aligned}
$$
The function $H_{x,\delta}$ is holomorphic on $\Omega$: Indeed, the first integral is entire in $s$, and $\Gamma(1-s)$ is holomorphic on $\Omega$. For the second integral, the assertion follows from the local uniform boundedness of $G(s,2\log(x)u)$ established above and from $Q\in L^1((0,1))$. Finally, also the third integral is entire: In fact, if $K\subset\mathbb C$ is compact then, uniformly for $s\in K$ and $\delta\leq u\leq1$,
$$
|\operatorname{Li}_s(e^{-2\log(x)u})|
\leq
\sum_{k\geq1}e^{-2\log(x)\delta k}k^{-\min_{s\in K}\Re e(s)}<\infty.
$$
Thus, the right-hand side of \eqref{eq:DL-cont} is meromorphic on $\Omega$ and agrees with $D_x(s)$ on the non-empty strip $1<\Re e(s)<2$. It therefore gives the desired meromorphic continuation of $D_x(s)$ to $\Omega$.\\\\
Let now $\rho$ be a non-trivial zero of $\zeta(s)$ of multiplicity $m_\rho$. Since $0<\Re e(\rho)<1$, we have $\rho\in\Omega$, and $H_{x,\delta}$ is holomorphic at $\rho$. Moreover,
$$
\Gamma(1-\rho)\neq0,\qquad (2\log(x))^{\rho-1}\neq0,\qquad \rho-1\neq0.
$$
Hence the first term on the right-hand side of \eqref{eq:DL-cont} has a pole of order $m_\rho$ at $s=\rho$. Since the second term is holomorphic there, $D_x(s)$ has a pole of order $m_\rho$ at $s=\rho$.\\\\
We now turn to the partial sums, which we consider as a function in a real variable $T\geq1$:
$$
\Sigma R^{triv}_T(x)=\sum_{1\leq k\leq T}a_k = \sum_{1\leq k\leq T}\ERi(-2k\log(x)).
$$
Suppose, for contradiction, that $\Sigma R^{triv}_T(x)=O(T^\theta)$, for some $\theta<\Theta$. By the definition of $\Theta$, there exists a non-trivial zero $\rho$ of $\zeta(s)$ such that $\beta:=\Re e(\rho)>\theta$. For $\Re e(s)>1$, partial summation gives
$$
\sum_{k\leq N}\frac{a_k}{k^s} = \Sigma R^{triv}_N(x)N^{-s} + s\int_1^N \Sigma R^{triv}_T(x)T^{-s-1}\,dT.
$$
Since $|a_k|\leq\|Q\|_1$, we have $\Sigma R^{triv}_N(x)=O(N)$. Therefore, if $\Re e(s)>1$,
$$
\Sigma R^{triv}_N(x)N^{-s}=O(N^{1-\Re e(s)})=o(1).
$$ Hence 
\begin{equation}\label{eq:DL-partial-summation}
D_x(s)=s\int_1^\infty \Sigma R^{triv}_T(x)T^{-s-1}\,dT, \qquad \Re e(s)>1.
\end{equation}
Under the assumed bound $\Sigma R^{triv}_T(x)=O(T^\theta)$, however, the right-hand side of \eqref{eq:DL-partial-summation} converges absolutely and locally uniformly throughout the half-plane $\Re e(s)>\theta$. Indeed, if $K$ is a compact subset of this half-plane, then there is an $\varepsilon>0$ such that
$$
\Re e(s)\geq\theta+\varepsilon
\qquad \forall s\in K,
$$
and hence, uniformly for $s\in K$ and $T\geq1$, $\left|\Sigma R^{triv}_T(x)T^{-s-1}\right| \ll_K T^{-1-\varepsilon}$. Consequently,
$$
s\longmapsto s\int_1^\infty \Sigma R^{triv}_T(x)T^{-s-1}\,dT
$$
is holomorphic on $\Re e(s)>\theta$ and provides a holomorphic continuation of the Dirichlet series $D_x$ to this half-plane. On the other hand, \eqref{eq:DL-cont} gives a meromorphic continuation of $D_x$ to $\Omega=\mathbb C\setminus\{1,2,3,\ldots\}$. The two continuations agree on the non-empty open strip $1<\Re e(s)<2$. By uniqueness of meromorphic continuation, they therefore agree on the connected domain $\{s\in\mathbb C:\Re e(s)>\theta\}\setminus\{1,2,3,\ldots\}$. Since $\Re e(\rho)=\beta>\theta$ and $\rho$ is not a positive integer, the continuation furnished by \eqref{eq:DL-partial-summation} is holomorphic at $s=\rho$. This contradicts the fact proved above that the meromorphic continuation \eqref{eq:DL-cont} has a pole at $s=\rho$. Thus,
$$
\Sigma R^{triv}_T(x)\neq O(T^\theta)
$$
for every $\theta<\Theta$. Taking $\theta=0$ gives $\Sigma R^{triv}_T(x)\neq O(1)$, whence $\limsup_{T\to\infty}|\Sigma R^{triv}_T(x)|=\infty$. In particular, the sequence of partial sums is unbounded, and hence the series $\sum_{k\geq1}R(x^{-2k})$ does not converge.
\end{proof}

\begin{remark}\label{rem:trivsharp}
The lower bound in the preceding theorem is again optimal. Indeed, for every fixed $x>1$ and every $\varepsilon>0$,
$$
 \Sigma R_N^{\rm triv}(x) = O_{x,\varepsilon}\bigl(N^{\Theta+\varepsilon}\bigr)
$$
as $N\to\infty$. To see this, let us abbreviate $G_{N}(u):=\sum_{k=1}^N e^{-2k \log(x)u}$. Then, Prop.\ \ref{prop:kernel} gives 
$$\Sigma R_N^{\rm triv}(x)=\sum_{k=1}^N\ERi(-2k \log(x))=  \int_0^1Q(u)G_{N}(u)\,du.$$ 
Since $G_{N}(u)\leq N$ and $G_{N}(u) \leq \sum_{k\geq1}e^{-2k \log(x)u} =\frac{1}{e^{2\log(x)u}-1}\leq \frac{1}{2\log(x)u},$
we have
$$
   G_{N}(u)\ll_x\min\{N,u^{-1}\}.
$$
Let us first assume that $\Theta=1$. Then readily $\left|\Sigma R_N^{\rm triv}(x)\right| \leq N\int_0^1|Q(u)|\,du =O_x(N)$ and the claim follows. Hence, assume now that $\Theta<1$. Then, for every $\eta>0$ satisfying $\Theta+\eta<1$, the estimate obtained above gives $Q(u)\ll_\eta u^{-\Theta-\eta}$ as $u\to0^+$. So, splitting the integral at $u=N^{-1}$, we obtain
$$
   \left|\Sigma R_N^{\rm triv}(x)\right| \ll_{x,\eta} N\int_0^{1/N}u^{-\Theta-\eta}\,du + \int_{1/N}^1u^{-\Theta-\eta-1}\,du\ll_{x,\eta}N^{\Theta+\eta}.
   $$
Choosing $\eta<\varepsilon$ proves the asserted estimate. As in the case of the non-trivial zeros, combining this bound with the preceding theorem, cf.\ Thm.\ \ref{prop:trivial-zeros}, yields
$$
   \limsup_{N\to\infty}\frac{\log\!\left(2+\left|\Sigma R_N^{\rm triv}(x)\right|\right)}{\log N} =\Theta.
$$
\end{remark}

\begin{remark}\label{rem:final}
On a final note, let us point out that summing all zeros of $\zeta(s)$ together, i.e., summing jointly the trivial and the non-trivial zeros, ordered by their distance to $1/2$ (that is, essentially by their absolute value), also yields a divergent series. More precisely, let 
$$\Sigma R^*_T(x):=\sum_{\substack{\zeta(\rho)=0\\ |\rho-\frac12|\leq T}}m_\rho\,\ERi(\rho\log(x)).$$
Then, again a soft modification of the strategy used in the proof of the Main Theorem, analogous to the proof of Thm.\ \ref{prop:trivial-zeros}, shows that $\Sigma R^*_T(x)$ is not $O(T^\theta)$ for any given $x>1$ and any $\theta<\Theta$ as $T\to\infty$. 
\end{remark}

\end{document}